\documentclass{article}
\usepackage{amsthm, amsmath, amssymb, amsfonts, mathrsfs, enumitem, tikz-cd, mathtools}

\usepackage{tikz}
\usepackage[framemethod=TikZ]{mdframed}
\usepackage[most]{tcolorbox}
\usepackage[colorlinks=true,linkcolor=black,anchorcolor=black,citecolor=black,filecolor=black,menucolor=black,runcolor=black,urlcolor=blue]{hyperref}
\usetikzlibrary{external, fit, shapes, decorations.pathreplacing, calligraphy, arrows.meta, positioning}
\usetikzlibrary{decorations.markings,arrows}

\usepackage{geometry}
\newcommand{\oo}{\infty}
\newcommand{\R}{\mathbb{R}}
\newcommand{\C}{\mathbb{C}}
\newcommand{\N}{\mathbb{N}}
\newcommand{\E}{\mathbb{E}}
\newcommand{\bin}{\text{Bin}}

\newtheoremstyle{plain}{3mm}{3mm}{\slshape}{}{\bfseries}{.}{.5em}{}
\newtheoremstyle{definition}{2mm}{2mm}{}{}{\bfseries}{.}{.5em}{}
\theoremstyle{plain} 
	
\newtheorem{theorem}{Theorem}[section]
\newtheorem{claim}{Claim}[section]

\newtheorem{question}[theorem]{Question}

\newtheorem{lemma}[theorem]{Lemma}

\theoremstyle{definition} 
\newtheorem{definition}[theorem]{Definition}
\newtheorem{remark}[theorem]{Remark}
\newtheorem{example}[theorem]{Example}

\theoremstyle{plain} 
\newcounter{MainTheoremCounter}

\newtheorem{maintheorem}[MainTheoremCounter]{Theorem}

\theoremstyle{plain}
\newtheorem*{namedtheorem}{\namedthmname}
\newcounter{namedtheorem}
	
\author{By~{\scshape Andy Liu and Michael Reilly}}
\date{\small \today}
\title{\bfseries A Composition Theorem for Binomially Weighted Averages}

\begin{document}

\maketitle



\begin{abstract}
We study binomially weighted summation methods given by 
\[
(x_n)_{n\in \N} \mapsto \left(\sum_{k=0}^n\binom{n}{k}r^k(1-r)^{n-k}x_k\right)_{n\in \mathbb{N}}
\]
for $r\in (0,1)$, and their behavior under composition with summation methods of the form 
\[
(x_n)_{n\in \N} \mapsto \left(\sum_{k=0}^n\lambda_k x_{n-k}\right)_{n\in \mathbb{N}}.
\]
Our main result shows that if the binomially weighted averages of a sequence $(x_n)_{n\in \N}$
converge to a limit then the binomially weighted averages of the sequence $\left(\sum_{k=0}^n\lambda_kx_{n-k}\right)_{n\in \mathbb{N}}$ converge to the same limit whenever $(\lambda_n)_{n\in\N}$ is an absolutely summable sequence with $\sum_{k=0}^{\infty}\lambda_k = 1$. This result disproves a theorem appearing in the literature. Additionally, we discuss applications and extensions of our main result to compositions with weighted Ces\`aro averages.

\end{abstract}

\section{Introduction}
In this paper we consider a method of summation which sends a sequence $(x_n)_{n\in \N}$ to averages of the form
\begin{equation}\label{eq:bin_avg}
\frac{1}{2^N}\sum_{n=0}^N\binom{N}{n}x_n
\end{equation}
for $N\in \N=\{0,1,\dots\}$. These averages were introduced by Euler in order to increase the rate of convergence of certain alternating series (see \cite[p. 8]{Hardy}). More generally, for $r\in (0,1)$ we define the \emph{$N$-th $r$-binomial average}\footnote{In \cite[p. 8]{Hardy} the map $(x_n)_{n\in \N}\mapsto \lim_{N\to\oo}\frac{1}{2^N}\sum_{n=0}^N\binom{N}{n}x_n$ is referred to as the \emph{Euler method} and denoted by $(E,1)$.} of the sequence $\vec{x}=(x_n)_{n\in \N}$ and we use the notation
\[
\E_{n\leq N}^{\bin(r)}(\vec{x})= \E_{n\leq N}^{\bin(r)}x_n = \sum_{n=0}^N\binom{N}{n}r^n(1-r)^{N-n}x_n.
\]
Note that $\E_{n\leq N}^{\bin(1/2)}x_n$ is equal to (\ref{eq:bin_avg}). These averages are studied in \cite{Gajser16} and in particular, a composition theorem for binomial averages and Ces\`aro averages is shown \cite[Lemma 5.8]{Gajser16}. Our goal is to show a analogous composition theorem holds when Ces\`aro averages are replaced by the averaging method introduced in \cite{natarajan2002}. We prove the following result.

\begin{maintheorem}\label{thm:main}
    Let $(\lambda_n)_{n\in \N}$ be a sequence of complex numbers such that $\sum_{n=0}^{\oo}|\lambda_n|<\oo$. Pick $r\in (0,1)$, let $L\in \C$, and let $(x_n)_{n\in \N}$ be a complex-valued sequence such that $\lim_{N\to\oo}\E_{n\leq N}^{\bin(r)}x_n= L$.
   Then
\begin{equation}\label{eq:main_thm_eq}
        \lim_{N\to\oo}\E_{n\leq N}^{\bin(r)}\left(\sum_{k=0}^n\lambda_kx_{n-k}\right) = L\cdot \left(\sum_{n=0}^{\oo}\lambda_n\right).
    \end{equation}
\end{maintheorem}
Notably, a theorem having the same hypothesis as Theorem \ref{thm:main} but a different value of the limit appears in \cite{Natarajan2013}.

\subsection*{Structure of the Paper}

In Section \ref{section:2} we provide a counterexample to an incorrect version of Theorem \ref{thm:main} contained in \cite{Natarajan2013} and we identify the error in the given proof. In Section \ref{section:3} we provide a proof of Theorem \ref{thm:main}. Lastly, in Section \ref{section:4} we give an application of Theorem \ref{thm:main} to the theory of weighted averaging methods.

\subsection*{Acknowledgments}
This work was done as a part of the Cycle undergraduate research program at The Ohio State University. The authors would like to thank this program for its generous non-financial support, without which this work would not have been possible.

\section{Counterexample to \cite[Theorem 2.3]{Natarajan2013}}\label{section:2}

In this section, we discuss an incorrect version of Theorem \ref{thm:main} that appears in the literature. The following theorem is given in \cite{Natarajan2013}.

\begin{claim}[{\cite[Theorem 2.3]{Natarajan2013}}]\label{thm:wrong}
Let $(\lambda_n)_{n\in \N}$ be a sequence of complex numbers such that $\sum_{n=0}^{\oo}|\lambda_n|<\oo$. Pick $r\in (0,1)$ and let $(x_n)_{n\in \N}$ be a complex-valued sequence such that $\lim_{N\to\oo}\E_{n\leq N}^{\bin(r)}x_n= L$. Then,
    \begin{equation}\label{eq:incorrect}
    \lim_{N\to\oo}\E_{n\leq N}^{\bin(r)}\left(\sum_{k=0}^n\lambda_kx_{n-k}\right) =L \cdot \left( \lambda_0 + \sum_{n=1}^\infty \lambda_n r^{n - 1}\right)
    \end{equation}
\end{claim}

However, this result is incorrect. The fundamental issue is that limits of binomially weighted averages of the form $\E_{n\leq N}^{\bin(r)}$ are independent of $r$. Indeed, as shown in \cite[Theorem 4.2]{Gajser16} if $(x_n)$ is a sequence such that $\lim_{N\to\oo}\E_{n\leq N}^{\bin(r)}x_n$ converges, then for any $r'<r$ the limit $\lim_{N\to\oo}\E_{n\leq N}^{\bin(r')}x_n$ converges to the same value. Therefore, the right side of equation (\ref{eq:incorrect}) cannot depend on $r$.

To give a more concrete counterexample, let $r=1/2$, let $L=1$, let $x_n = 1$ for all $n$, and let $(\lambda_k)_{k=0}^{\oo}$ be the sequence $(1/3,1/3,1/3,0,0,\dots)$. For each $N\in \N$, we have that $\E_{n\leq N}^{\bin(1/2)}x_n = \frac{1}{2^N}\sum_{n=0}^N\binom{N}{n}\cdot 1 = 1$, thus $\lim_{N\to\oo}\E_{n\leq N}^{\bin(1/2)}x_n=L$. For $N\geq 3$, the sum $\sum_{k=0}^N\lambda_k x_{N-k}$ is 
$$
\sum_{k=0}^N\lambda_k x_{N-k}=\frac{x_N}{3}+\frac{x_{N-1}}{3}+\frac{x_{N-2}}{3}+0\cdot x_{N-3}+0\cdot x_{N-4}+\cdots =\frac{1}{3}+\frac{1}{3}+\frac{1}{3}=1.
$$
Therefore, the sequence $\left(\sum_{k=0}^N\lambda_k x_{N-k}\right)_{N\in \N}$ is equal to $(1/3,2/3,1,1,1,\dots)$. We then calculate the left side of equation (\ref{eq:incorrect})
\begin{align*}
 \lim_{N\to\oo}\E_{n\leq N}^{\bin(r)}\left(\sum_{k=0}^n\lambda_kx_{n-k}\right)  =&  \lim_{N\to\oo}\frac{1}{2^N}\sum_{n=0}^N\binom{N}{n}\left(\sum_{k=0}^n\lambda_kx_{n-k}\right)  \\=&\lim_{N\to\oo}\frac{1}{2^N}\left(\binom{N}{0}\cdot \frac{1}{3}+\binom{N}{1}\cdot \frac{2}{3}+\sum_{n=2}^N\binom{N}{n}\right)
\end{align*}
which converges to $1$. However, the right side of equation (\ref{eq:incorrect}) is equal to 
$$
L \cdot \left( \lambda_0 + \sum_{n=1}^\infty \lambda_n r^{n - 1}\right) = 1\cdot\left(\frac{1}{3}+\frac{1}{3}\cdot \left(\frac{1}{2}\right)^0+\frac{1}{3}\cdot \left(\frac{1}{2}\right)^1+0\cdot \left(\frac{1}{2}\right)^2+0\cdot \left(\frac{1}{2}\right)^3+\cdots\right) = \frac{5}{6}.
$$

Hence the two sides of (\ref{eq:incorrect}) do not agree. Therefore, Claim \ref{thm:wrong} cannot be correct and the proof in \cite{Natarajan2013} must contain an error. Indeed, the proof in \cite[pg. 192-193]{Natarajan2013} uses the following identity.
\begin{claim}\label{claim:wrong_binom_identity}
For any $n,k,\ell\in \N$ with $\ell<n$ and $0\leq k\leq n-\ell$,
\begin{equation}\tag{$\star$}\label{eq:wrong}
   \sum_{j=k}^{n-\ell}(-1)^{j-k}\binom{n}{j+\ell}\binom{j}{k} = 1.
\end{equation}
\end{claim}

In \cite[Equation (2.4)]{Natarajan2013} the special case $\ell=1$ appears and it is assumed that it holds for all $\ell$. However, (\ref{eq:wrong}) is only true when $\ell=1$. What is true is the following.
\begin{lemma}
For any $n,k,\ell\in \N$ with $\ell<n$ and $0\leq k\leq n-\ell$,
\begin{equation}\tag{$\star\star$}\label{eq:right}
        \sum_{j=k}^{n-\ell}(-1)^{j-k}\binom{n}{j+\ell}\binom{j}{k} = \binom{n-(k+1)}{\ell-1}
    \end{equation}
\end{lemma}

\begin{proof}

    Using the fact that $
    \binom{n}{k} = (-1)^k \binom{k - n - 1}{k}$
    we have
    \[
    \sum_{j=k}^{n-\ell}\binom{n}{j+\ell}\binom{j}{k}(-1)^{j-k} = \sum_{j=k}^{n-\ell}\binom{n}{j+\ell}\binom{j - k - j - 1}{j - k} = \sum_{j=k}^{n-\ell}\binom{n}{j+\ell}\binom{-k - 1}{j - k}.
    \]
    By replacing $j$ with $j+k$, this sum becomes
    \[
    \sum_{j=k}^{n-\ell}\binom{n}{j+\ell}\binom{-k - 1}{j - k} = \sum_{j=0}^{n-\ell - k}\binom{n}{j + k +\ell}\binom{-k - 1}{j}= \sum_{j=0}^{n-\ell - k}\binom{n}{n-(j + k +\ell)}\binom{-k - 1}{j}.
    \]
    Recall the Chu-Vandermonde identity, which states that for any $r\in \N$ and any $a,b\in \C$ we have
    \begin{equation*}
    \sum_{j = 0}^r  \binom{a}{r-j} \binom{b}{j}= \binom{a + b}{r}.
    \end{equation*}
    Then taking $r=n-\ell-k$, $a=n$, $b=-(k+1)$, we can rewrite
    \[
    \sum_{j=0}^{n-\ell - k}\binom{n}{(n-k-\ell)-j}\binom{-k - 1}{j} = \binom{n - (k + 1)}{n - \ell - k}= \binom{n - (k + 1)}{n - (k + 1) - (n - \ell - k)} = \binom{n - (k + 1)}{\ell - 1}.
    \]
    In total we have
    \[
    \sum_{j=k}^{n-\ell}\binom{n}{j+\ell}\binom{j}{k}(-1)^{j-k} = \binom{n-(k+1)}{\ell-1}.
    \]
\end{proof}

We note that when $\ell=1$, equations (\ref{eq:wrong}) and (\ref{eq:right}) agree since $\binom{n-(k+1)}{1-1} = 1$.

\section{Proof of Theorem \ref{thm:main}}\label{section:3}

The goal of this section is to prove Theorem \ref{thm:main}. Our strategy is to show that binomial averages are asymptotically shift-invariant in the following sense. 
Let $T$ be the right-shift operator on sequences, defined by $T\vec{y} = (0,y_0,y_1,\dots)$ for $\vec{y} = (y_0,y_1,\dots)$. We will show that for each $k\in \N$:
\begin{equation}\label{eq:shift_inv}
\text{ If } \lim_{N\to\oo}\E_{n\leq N}^{\bin(r)}\vec{x}=L \text{ then }\lim_{N\to\oo}\E_{n\leq N}^{\bin(r)}T^k\vec{x}=L .
\end{equation}
Then we will manipulate the left-hand side of equation (\ref{eq:main_thm_eq}) so that it becomes a sum containing terms of the form $\E_{n\leq N}^{\bin(r)}T^k\vec{x}$. In order to prove (\ref{eq:shift_inv}) we will begin with an important algebraic fact about binomial averages.

\begin{lemma} 
Let $(x_n)_{n\in \N}$ be a complex-valued sequence. Then for each $r\in (0,1)$ and for each $N\geq 1$ we have that
    \begin{equation}\label{eq:shift_invariant}
        r\cdot \E_{n\leq N}^{\bin(r)}(x_{n+1})+(1-r)\cdot \E_{n\leq N}^{\bin(r)}(x_{n}) = \E_{n\leq N+1}^{\bin(r)}(x_{n})
    \end{equation}
\end{lemma}
\begin{proof}
Fix $N\geq 1$, $r\in (0,1)$ and a sequence $(x_n)_{n\in \N}$. We will start with the left-hand side of equation (\ref{eq:shift_invariant}).
    \begin{align*}
        &r\cdot \E_{n\leq N}^{\bin(r)}(x_{n+1})+(1-r)\cdot \E_{n\leq N}^{\bin(r)}(x_{n}) \\
        =& r\cdot \sum_{n=0}^{N}\binom{N}{n}r^n(1-r)^{N-n}x_{n+1}+(1-r)\cdot \sum_{n=0}^{N}\binom{N}{n}r^n(1-r)^{N-n}x_{n}
        \\
        =& \sum_{n=0}^{N}\binom{N}{n}r^{n+1}(1-r)^{N-n}x_{n+1}+\sum_{n=0}^{N}\binom{N}{n}r^{n}(1-r)^{N+1-n}x_{n}.
        \end{align*}
        We can reindex the first sum then add in zero terms corresponding to $n=0$ in the first sum and $n=N+1$ term in the second in order to combine the two sums.
        \begin{align*}
        & \sum_{n=0}^{N}\binom{N}{n}r^{n+1}(1-r)^{N-n}x_{n+1}+\sum_{n=0}^{N}\binom{N}{n}r^{n}(1-r)^{N+1-n}x_{n}\\
        =&\sum_{n=1}^{N+1}\binom{N}{n-1}r^{n}(1-r)^{N-n+1}x_n+\sum_{n=0}^{N}\binom{N}{n}r^{n}(1-r)^{N+1-n}x_{n}\\
        =&\sum_{n=0}^{N+1}\binom{N}{n-1}r^{n}(1-r)^{N-n+1}x_n+\sum_{n=0}^{N+1}\binom{N}{n}r^{n}(1-r)^{N+1-n}x_{n}\\
        =&\sum_{n=0}^{N+1}\left(\binom{N}{n-1}+\binom{N}{n}\right)r^{n}(1-r)^{N-n+1}x_n
    \end{align*}
    Pascal's identity says that $\binom{N}{n-1}+\binom{N}{n} = \binom{N+1}{n}$ for all $n$, so this becomes
    \begin{align*}
         \sum_{n=0}^{N+1}\left(\binom{N}{n-1}+\binom{N}{n}\right)r^{n}(1-r)^{N-n+1}x_n  
         =\sum_{n=0}^{N+1}\binom{N+1}{n}r^{n}(1-r)^{N-n+1}x_n
         =\E_{n\leq N+1}^{\bin(r)}(x_n).
    \end{align*}
\end{proof}

\begin{remark}
    Rearranging equation (\ref{eq:shift_invariant}) gives
    \begin{equation}
        \E_{n\leq N}^{\bin(r)}(x_{n+1}) = \E_{n\leq N}^{\bin(r)}(x_{n})+r^{-1}\cdot (\E_{n\leq N+1}^{\bin(r)}(x_{n})-\E_{n\leq N}^{\bin(r)}(x_{n})).
    \end{equation}
    This makes it clear that if $ \lim_{N\to\oo}\E_{n\leq N}^{\bin(r)}(x_{n})$ converges to a limit then $\lim_{N\to\oo}\E_{n\leq N}^{\bin(r)}(x_{n+1})$ also converges to the same limit.
\end{remark}

\begin{lemma}
    Let $(x_n)_{n\in \N}$ be a complex-valued sequence, let $r\in (0,1)$ and let $k\in \N$. Then for each $N\geq 1$ we have that
    \begin{equation}\label{eq:induction_hypothesis}
        \E_{n\leq N+k}^{\bin(r)}(x_n) = r\cdot \sum_{i=0}^{k-1}(1-r)^{k-i-1}\cdot \E^{\bin(r)}_{n\leq N+i}(x_{n+1})+(1-r)^k\cdot \E_{n\leq N}^{\bin(r)}(x_n).
    \end{equation}
\end{lemma}
\begin{proof}

    We proceed by induction on $k$. When $k=1$, equation (\ref{eq:induction_hypothesis}) is the same as equation (\ref{eq:shift_invariant}). Now suppose that equation (\ref{eq:induction_hypothesis}) holds for some $k\geq 1$, and consider $\E_{n\leq N+k+1}^{\bin(r)}(x_n)$. Applying equation (\ref{eq:shift_invariant}) gives
    \begin{align*}
        \E_{n\leq N+k+1}^{\bin(r)}(x_n) = r\cdot \E_{n\leq N+k}^{\bin(r)}(x_{n+1})+(1-r)\cdot \E_{n\leq N+k}^{\bin(r)}(x_n).
    \end{align*}
We can apply the inductive hypothesis to $\E_{n\leq N+k}^{\bin(r)}(x_n)$ so that we have
 \begin{align*}
        &\E_{n\leq N+k+1}^{\bin(r)}(x_n) \\
        =& r\cdot \E_{n\leq N+k}^{\bin(r)}(x_{n+1})+(1-r)\cdot \E_{n\leq N+k}^{\bin(r)}(x_n)\\
        =& r\cdot \E_{n\leq N+k}^{\bin(r)}(x_{n+1})+(1-r)\cdot \left( r\cdot \sum_{i=0}^{k-1}(1-r)^{k-i-1}\cdot \E_{n\leq N+i}(x_{n+1})+(1-r)^k\cdot \E_{n\leq N}^{\bin(r)}(x_n)\right)\\
        =& r\cdot \E_{n\leq N+k}^{\bin(r)}(x_{n+1})+ r\cdot \sum_{i=0}^{k-1}(1-r)^{(k+1)-i-1}\cdot \E_{n\leq N+i}(x_{n+1})+(1-r)^{k+1}\cdot \E_{n\leq N}^{\bin(r)}(x_n)\\
        =&r\cdot \sum_{i=0}^{(k+1)-1}(1-r)^{(k+1)-i-1}\cdot \E_{n\leq N+i}(x_{n+1})+(1-r)^{k+1}\cdot \E_{n\leq N}^{\bin(r)}(x_n)
    \end{align*}
    which means that we are done.
\end{proof}

\begin{theorem}\label{thm:weak_shift_invariance}
Let $\vec{x}=(x_n)_{n\in \N}$ be a complex-valued sequence, let $T$ be the right-shift operator on sequences, and let $L\in \C$. If $ \lim_{N\to\oo}\E_{n\leq N}^{\bin(r)}(\vec{x})=L$ then $\lim_{N\to\oo}\E_{n\leq N}^{\bin(r)}(T\vec{x})=L$.
\end{theorem}
\begin{proof}

First note that by replacing $\vec{x}$ with $T\vec{x}$, equation (\ref{eq:induction_hypothesis}) becomes
\begin{equation}\label{eq:rec_relation_with_T}
        \E_{n\leq N+k}^{\bin(r)}(T\vec{x}) = r\cdot \sum_{i=0}^{k-1}(1-r)^{k-i-1}\cdot \E_{n\leq N+i}^{\bin(r)}(\vec{x})+(1-r)^k\cdot \E_{n\leq N}^{\bin(r)}(T\vec{x}).
\end{equation}

Pick $\epsilon>0$ and pick $N_0$ large enough so that $|\E_{n\leq N}^{\bin(r)}(\vec{x})-L|<\epsilon$ for all $N\geq N_0$. Consider the limit $\lim_{N\to\oo}\E_{n\leq N}^{\bin(r)}(T\vec{x}) = \lim_{k\to\oo}\E_{n\leq N_0+k}^{\bin(r)}(T\vec{x})$. By equation (\ref{eq:rec_relation_with_T}), this is 
\begin{align*}
    \lim_{k\to\oo}\E_{n\leq N_0+k}^{\bin(r)}(T\vec{x}) =& \lim_{k\to\oo}\left(r\cdot \sum_{i=0}^{k-1}(1-r)^{k-i-1}\cdot \E_{n\leq N_0+i}^{\bin(r)}(\vec{x})+(1-r)^k\cdot \E_{n\leq N_0}^{\bin(r)}(T\vec{x})\right).
\end{align*}
$\E_{n\leq N_0}^{\bin(r)}(T\vec{x})$ is a constant and so $\lim_{k\to\oo}(1-r)^k\cdot \E_{n\leq N_0}^{\bin(r)}(T\vec{x})=0$. Next, we will show that 
$$
\lim_{k\to\oo}r\cdot \sum_{i=0}^{k-1}(1-r)^{k-i-1}\cdot \E_{n\leq N_0+i}^{\bin(r)}(\vec{x}) =  \lim_{k\to\oo}r\cdot \sum_{i=0}^{k-1}(1-r)^{i}\cdot \E_{n\leq N_0+(k-i)}^{\bin(r)}(\vec{x})
$$
is equal to $L$.

For any $k$, we have
\begin{align}
   & \left| r\cdot \sum_{i=0}^{k-1}(1-r)^{i}\cdot \E_{n\leq N_0+(k-i)}^{\bin(r)}(\vec{x})-L\right|\\
    =& \left|  r\cdot \sum_{i=0}^{k-1}(1-r)^{i}\cdot \E_{n\leq N_0+(k-i)}^{\bin(r)}(\vec{x})-r\cdot \sum_{i=0}^{\oo}(1-r)^iL\right|\\
    \leq & r\sum_{i=0}^{k-1}(1-r)^i\left| \E_{n\leq N_0+(k-i)}^{\bin(r)}(\vec{x})-L\right|+r\sum_{i=k}^{\oo}(1-r)^i \cdot|L|\\
    <&  r\sum_{i=0}^{k-1}(1-r)^i\cdot \epsilon+o_{k\to\oo}(1)< \epsilon+o_{k\to\oo}(1).
\end{align}
Since $\epsilon$ is arbitrary, it follows that $\lim_{k\to\oo}r\cdot \sum_{i=0}^{k-1}(1-r)^{k-i-1}\cdot \E_{n\leq N_0+i}^{\bin(r)}(\vec{x}) = L$ and so $\lim_{N\to\oo}\E_{n\leq N}^{\bin(r)}T\vec{x}=L$.
    
\end{proof}

From this theorem, the implication in equation $(\ref{eq:shift_inv})$ follows by induction. So now we are ready to prove Theorem \ref{thm:main}.

\begin{proof}

Let $(\lambda_n)_{n\in \N}$ be a complex-valued sequence with $\sum_{n=0}^{\oo}|\lambda_n|<\oo$. Fix $r\in (0,1)$, $L\in \C$ and suppose that $\vec{x}=(x_n)_{n\in \N}$ is a complex-valued sequence with $\lim_{N\to\oo}\E_{n\leq N}^{\bin(r)}x_n = L$. We will show that 
\begin{equation}
    \lim_{N\to\oo}\E_{n\leq N}^{\bin(r)}\left(\sum_{k=0}^n\lambda_kx_{n-k}\right) = L.
\end{equation}

Let $T$ be the right-shift operator on sequences. Then for each $k\in \N$
\[
\E_{n\leq N}^{\bin(r)}T^k\vec{x} = \sum_{n=k}^N\binom{N}{n}r^n(1-r)^{N-n}x_{n-k}
\]
and so

\begin{align*}
    \E_{n\leq N}^{\bin(r)}\left(\sum_{k=0}^n\lambda_kx_{n-k}\right)  = \sum_{n=0}^N \binom{N}{n}r^n(1-r)^{N-n}\left(\sum_{k=0}^n\lambda_kx_{n-k}\right) &=\sum_{k=0}^N \lambda_k\left(\sum_{n=k}^N\binom{N}{n}r^n(1-r)^{N-n}x_{n-k}\right)\\
    &=\sum_{k=0}^N \lambda_k\cdot \E_{n\leq N}^{\bin(r)}T^k\vec{x}.
\end{align*}
We will show that
\begin{equation} \label{eq:dom_con}
\lim_{N\to\oo}\sum_{k=0}^N \lambda_k\cdot \E_{n\leq N}^{\bin(r)}T^k\vec{x} = \sum_{k=0}^{\oo}\lambda_k\cdot L.
\end{equation}

From Theorem \ref{thm:weak_shift_invariance}, $\lim_{N\to\oo}\E_{n\leq N}^{\bin(r)}T^k\vec{x}=L$ and so we will use the dominated convergence theorem to prove equation (\ref{eq:dom_con}). To this end, it suffices to find a constant $C$ such that $|\E_{n\leq N}^{\bin(r)}T^k\vec{x}|\leq C$ for all $k\geq 0$.

For each $k\in \N$, consider $\sup_{N\in \N}|\E_{n\leq N}^{\bin(r)}T^k\vec{x}|$. From equation (\ref{eq:rec_relation_with_T}) we have 
\[
|\E_{n\leq N+1}^{\bin(r)}(T\vec{x})|\leq r\cdot |\E_{n\leq N}^{\bin(r)}(\vec{x})|+(1-r)\cdot |\E_{n\leq N}^{\bin(r)}(T\vec{x})|
\]
and taking supremum over $N$ on both sides,
\[
\sup_{N\in \N}|\E_{n\leq N}^{\bin(r)}(T\vec{x})|\leq r\cdot\sup_{N\in \N} |\E_{n\leq N}^{\bin(r)}(\vec{x})|+(1-r)\cdot \sup_{N\in \N}|\E_{n\leq N}^{\bin(r)}(T\vec{x})|.
\]

Rearranging, we have 
\begin{equation}\label{eq:shift_and_sup}
\sup_{N\in \N}|\E_{n\leq N}^{\bin(r)}(T\vec{x})|\leq \sup_{N\in \N}|\E_{n\leq N}^{\bin(r)}(\vec{x})|.
\end{equation}

Taking $\vec{x} = T\vec{x}$ in equation (\ref{eq:shift_and_sup}), it follows by induction that $\sup_{N\in \N}|\E_{n\leq N}^{\bin(r)}(T^k\vec{x})|\leq \sup_{N\in \N}|\E_{n\leq N}^{\bin(r)}(\vec{x})|$ for each $k\in \N$. So we may take $C =\sup_{N\in \N}|\E_{n\leq N}^{\bin(r)}(\vec{x})|$. Then equation (\ref{eq:dom_con}) holds by the dominated convergence theorem, and so we are done.

\end{proof}

\section{Weighted Averaging Methods}\label{section:4}

In this section, we will investigate some consequences of Theorem \ref{thm:main}. To begin, we consider a special case of averages of the form $\sum_{k=0}^n\lambda_kx_{n-k}$.

\begin{definition}
 Let $W:\N\rightarrow \R$ be a function which increases to $\infty$ and let $(x_n)_{n\in \N}$ be a bounded sequence. For $N\in \N$, define the \emph{N-th $W$-weighted average} of $(x_n)_{n\in \N}$ by
\begin{equation}
    \E_{n\leq N}^Wx_n= \frac{1}{W(N)}\sum_{n=1}^N\Delta W(n) x_n,
\end{equation}
    where $\Delta W(n) = W(n)-W(n-1)$ for all $n$.
\end{definition}

\begin{lemma}\label{lem:weighted_is_lambda}
  Let $W:\N\rightarrow \R$ be a function which increases to $\infty$ and let $(x_n)_{n\in \N}$ be a bounded sequence. Suppose that $L=\lim_{N\to\oo}\frac{W(N-1)}{W(N)}$ exists with $L\in (0,1)$. For each $n\geq 0$, let $\lambda_n = L^{n}-L^{n+1}$. Then
    \begin{equation}\label{eq:weighted_is_lambda}
    \frac{1}{W(N)}\sum_{n=1}^N\Delta W(n) x_n = \sum_{k=0}^N \lambda_k x_{N-k}+o_{N\to\oo}(1).
    \end{equation}
\end{lemma}

\begin{proof}

Note that 
        \[
        L =\lim_{N\to\oo}\frac{W(N - 1)}{W(N)} =\lim_{N\to\oo} \frac{W(N - 2)}{W(N - 1)} =  \dots = \lim_{N\to\oo}\frac{W(N - m)}{W(N - m + 1)}.
        \]
        Hence,
        \begin{equation}\label{eq:L_eq}
        \lim_{N\to\oo}\frac{W(N-k)}{W(N)} = \lim_{N\to\oo}\frac{W(N - k)}{W(N - k + 1)} \cdot \frac{W(N - k + 1)}{W(N- k + 2)}\cdots \frac{W(N -2)}{W(N-1)} \cdot \frac{W(N -1)}{W(N)}= L^k 
        \end{equation}

   Let $\epsilon>0$. To show that (\ref{eq:weighted_is_lambda}) holds, we will pick $M\in \N$ and apply the triangle inequality,
\begin{align}\label{eq:triangle_inequality}
        \left| \frac{1}{W(N)} \sum_{n = 1}^N \Delta W(n) x_n - \sum_{k = 0}^N \lambda_kx_{N - k} \right| \leq A(N,M)+B(N,M)+C(N,M)
\end{align}
where 
\begin{align*}
    A(N,M) =& \left| \frac{1}{W(N)} \sum_{n = 1}^N \Delta W(n) x_n - \frac{1}{W(N)} \sum_{n = N -M}^N \Delta W(n) x_n \right|\\
    B(N,M) =& \left| \frac{1}{W(N)} \sum_{n = N - M}^N \Delta W(n) x_n - \sum_{k = 0}^{M} \lambda_kx_{N - k} \right|\\
    C(N,M) =& \left|\sum_{k = 1}^{M} \lambda_kx_{N - k} - \sum_{k = 1}^{N} \lambda_kx_{N - k} \right|.
\end{align*}
    We will show that when $M$ is sufficiently large, $A(N,M) < \epsilon/2 + o_{N\to\oo}(1)$, $B(N,M) = o_{N\to\oo}(1)$, and $C(N,M)<\epsilon/2+o_{N\to\oo}(1)$, so that the right hand side of (\ref{eq:triangle_inequality}) is bounded by $\epsilon+o_{N\to\oo}(1)$.

First consider $A(N,M)$.
        \begin{align*}
           A(N,M) =& \left|\frac{1}{W(N)}\sum_{n=1}^N\Delta W(n)x_n - \frac{1}{W(N)}\sum_{n=N-M}^N\Delta W(n)x_n\right| = \left|\frac{1}{W(N)}\sum_{n=1}^{N-M-1}\Delta W(n)x_n \right|\\
            \leq & \sup_{n\in \N}|x_n|\cdot \frac{1}{W(N)}\sum_{n=1}^{N-M-1}\Delta W(n)= \sup_{n\in \N}|x_n|\cdot \frac{W(N-M-1)}{W(N)} = \sup_{n\in \N}|x_n|\cdot L^{M+1}+o_{N\to\oo}(1).
        \end{align*}
        $L<1$ by assumption and so for sufficiently large $M$, $\sup_{n\in \N}|x_n|\cdot L^{M+1} <\epsilon/2$. In order to bound $B(N,M)$, we can rewrite $\frac{1}{W(N)} \sum_{n = N - M}^N \Delta W(n) x_n$ and apply equation (\ref{eq:L_eq}) to obtain
        \begin{align*}
 \frac{1}{W(N)} \sum_{n = N - M}^N \Delta W(n) x_n = \sum_{k = 0}^{M } \frac{\Delta W(N - k)}{W(N)}x_{N-k} 
        =&\sum_{k = 0}^{M } (\lambda_k+o_{N\to\oo}(1))x_{N-k}= \sum_{k = 0}^{M } \lambda_k x_{N-k}  +o_{N\to\oo}(1),
        \end{align*}
        which means that $B(N,M) = o_{N\to\oo}(1)$.

Lastly, we will bound $C(N,M)$. Note that 
\begin{align*}
    &\left|\sum_{k = 1}^{M} \lambda_kx_{N - k} - \sum_{k = 1}^{N} \lambda_kx_{N - k}\right| = \left|\sum_{k = M+1}^{N} \lambda_kx_{N - k}\right|\\
    = &\left|\sum_{k = M+1}^{N} \left(L^k-L^{k+1}\right)x_{N - k} \right|=  \left|L^{M+1}\sum_{k = 0}^{N-M-1} \left(L^k-L^{k+1}\right)x_{k}\right| \\
    =&\left|L^{M+1}(x_0-L^{N-M}x_{N-M-1})\right|\leq \sup_{n\in \N}|x_n|\cdot L^{M+1}+\sup_{n\in \N}|x_n|\cdot L^{N-M}.
\end{align*}
       When $M$ is large enough, $\sup_{n\in \N}|x_n|\cdot L^{M+1}<\epsilon/2$, and for any fixed $M$ we have $\sup_{n\in \N}|x_n|\cdot L^{N-M} = o_{N\to\oo}(1)$. Thus we have shown that $C(N,M)<\epsilon/2+o_{N\to\oo}(1)$. This concludes the proof.
\end{proof}

\begin{example}
Fix any bounded sequence $(x_n)_{n\in \N}$. Let $W(N) = a^N$ for $a>1$. Then $\lim_{N\to\oo}\frac{W(N-1)}{W(N)} = \frac{1}{a}>0$, and taking $\lambda_{k} = \frac{\Delta W(N-k)}{W(N)} = \frac{1}{a^k}-\frac{1}{a^{k+1}}$ we have the exact equality
\begin{equation}
\frac{1}{W(N)}\sum_{n=1}^N\Delta W(n)x_n = \sum_{k=0}^N \lambda_k x_{N-k}.
\end{equation}

Now take $W(N) = e^N$ so that $\Delta W(n) = e^n-e^{n-1}$ and $\lim_{N\to\oo}\frac{W(N-1)}{W(N)} = \frac{1}{e}$. Taking $\lambda_n = \frac{1}{e^n}-\frac{1}{e^{n+1}}$, Lemma \ref{lem:weighted_is_lambda} says that 
\begin{equation}
  \frac{1}{e^N}\sum_{n=1}^N\left(e^n-e^{n-1}\right)\cdot x_n = \sum_{k=0}^N \left(\frac{1}{e^k}-\frac{1}{e^{k+1}}\right) x_{N-k}+o_{N\to\oo}(1).
\end{equation}

\end{example}

Lemma \ref{lem:weighted_is_lambda} gives a way to reformulate a special case of our main theorem.
\begin{theorem}
    Let $r\in (0,1)$, $L\in \C$ and let $W$ be an increasing function with $\lim_{N\to\oo}W(N)=\oo$, such that $\lim_{N\to\oo}\frac{W(N-1)}{W(N)}$ exists and is positive. Suppose that $(x_n)_{n\in \N}$ is a bounded sequence such that $\lim_{N\to\oo}\E_{n\leq N}^{\bin(r)}x_n = L$. Then 
    \begin{equation}\label{eq:main_thm_W}
        \lim_{N\to\oo}\E_{n\leq N}^{\bin(r)}\E_{k\leq n}^Wx_k = L.
    \end{equation}
\end{theorem}

We can compare equation (\ref{eq:main_thm_W}) with the following result in \cite{Gajser16}.

\begin{theorem}[{\cite[Lemma 5.8]{Gajser16}}]\label{thm:Gajser}
        Let $r\in (0,1)$ and let $W(N) = N$. Suppose that $(x_n)_{n\in \N}$ is a bounded sequence such that $\lim_{N\to\oo}\E_{n\leq N}^{\bin(r)}x_n = L$. Then 
    \begin{equation}\label{eq:Gasjer_eq}
        \lim_{N\to\oo}\E_{n\leq N}^{\bin(r)}\E^W_{m\leq n}x_m=L.
    \end{equation}
\end{theorem}

This suggests the following question:

\begin{question}
    Let $r\in (0,1)$, $L\in \C$ and let $(x_n)_{n\in \N}$ be a bounded sequence such that $\lim_{N\to\oo}\E_{n\leq N}^{\bin(r)}x_n = L$.   
    Which functions $W$ satisfy 
    \[
    \lim_{N\to\oo}\E_{n\leq N}^{\bin(r)}\E^W_{m\leq n}x_m=L?
    \]
    Does this statement hold for every function $W$ which increases to $\oo$? If not, is there a regularity condition on $W$ which is sufficient for the statement to hold?
\end{question}

\bibliographystyle{unsrt}
\bibliography{references}

@article{Natarajan2013,
url = {https://doi.org/10.1524/anly.2013.1193},
title = {A Product Theorem for the Euler and the Natarajan Methods of Summability},
author = {P. Natarajan},
pages = {189--196},
volume = {33},
number = {2},
journal = {Analysis},
doi = {doi:10.1524/anly.2013.1193},
year = {2013},
lastchecked = {2025-04-10}
}

@article{Gajser16,
author = {D. Gajser},
year = {2016},
month = {04},
pages = {393-410},
title = {On convergence of binomial means, and an application to finite Markov chains},
volume = {10},
journal = {Ars Mathematica Contemporanea},
doi = {10.26493/1855-3974.705.56d}
}

@book{Hardy,
    author = {G. Hardy},
    title = {Divergent Series},
    publisher = {The Clarendon Press},
    year = {1949}
}

@article{natarajan2002,
  title={Some properties of the $Y$-method of summability in complete ultrametric fields},
  author={Natarajan, P.},
  journal={Annales math{\'e}matiques Blaise Pascal},
  volume={9},
  number={1},
  pages={79--84},
  year={2002}
}

\end{document}